\documentclass[revision]{FPSAC2021}

\usepackage[utf8]{inputenc}
\usepackage[T1]{fontenc}
\usepackage{amssymb}
\usepackage{calc}
\usepackage[french,english]{babel}

\usepackage{amsthm}
\usepackage{amsfonts}
\usepackage{amssymb}
\usepackage{tikz}
\usepackage{color}
\usepackage{float}
\usepackage[makeroom]{cancel}
\usepackage{ytableau}
\usepackage{mathdots}
\usepackage{multicol}
\usepackage{mathtools}
\usepackage{stmaryrd}
\usepackage{multirow}
\usepackage{pdflscape}

\usepackage{lipsum}

\newtheorem{teo}{Theorem}[section]
\newtheorem{cor}[teo]{Corollary}
\newtheorem{prop}[teo]{Proposition}
\newtheorem{lema}[teo]{Lemma}

\theoremstyle{definition}

\newtheorem{defin}[teo]{Definition}

\theoremstyle{remark}
\newtheorem{obs}[teo]{Remark}
\newtheorem{ex}[teo]{Example}

\DeclareMathSymbol{\shortminus}{\mathbin}{AMSa}{"39}

\definecolor{lgray}{rgb}{0.65, 0.65, 0.65}
\definecolor{llblue}{HTML}{cce6ff}
\numberwithin{equation}{section}
\title[]{Shifted Bender--Knuth moves and a shifted Berenstein--Kirillov group}

\author[Inês Rodrigues]{Inês Rodrigues\thanks{\href{mailto:imarrodrigues@fc.ul.pt}{imarrodrigues@fc.ul.pt}.}\addressmark{1}}

\address{\addressmark{1} Center for Functional Analysis, Linear
Structures and Applications, Faculty of Sciences, University of Lisbon, Portugal}

\received{\today}

\revised{}

\abstract{The Bender--Knuth involutions on Young tableaux are known to coincide with the tableau switching on two adjacent letters, together with a swapping of those letters. Using the shifted tableau switching due to Choi--Nam--Oh (2019), we introduce a shifted version of the Bender--Knuth operators and define a shifted version of the Berenstein--Kirillov group. The actions of the cactus group, due to the author, and of the shifted Berenstein--Kirillov group on the Gillespie--Levinson--Purbhoo straight-shaped shifted tableau crystal (2017, 2020) coincide. Following the works of Halacheva (2016, 2020), and Chmutov--Glick--Pylyavskyy (2016, 2020), on the relation between the actions of the Berenstein--Kirillov group and the cactus group on the crystal of straight-shaped Young tableaux, we show that the shifted Berenstein--Kirillov group is isomorphic to a quotient of the cactus group. Not all the known relations that hold in the classic Berenstein--Kirillov group need to be satisfied by the shifted Bender--Knuth involutions, but the ones implying the relations of the cactus group are verified. Hence we have an alternative presentation for the cactus group via the shifted Bender--Knuth involutions.
}


\begin{document}

\maketitle

\ytableausetup{smalltableaux}

\section{Introduction}
The Bender--Knuth moves $t_i$ are well known involutions on semistandard Young tableaux, that act on adjacent letters $i$ and $i+1$ reverting their multiplicities, and leaving the others unchanged
\cite{BeKn72}. The tableau switching, due to Benkart, Sottile and Stroomer 
\cite{BSS96} is an algorithm on pairs of semistandard Young tableaux $(S,T)$, with $T$ extending $S$, that moves one through the other, obtaining a pair $(^S T,S_T)$ component-wise Knuth equivalent to $(T,S)$. The tableau switching of two adjacent letters, together with a swapping of those letters, coincides with the classic Bender--Knuth involutions 
\cite{BSS96}. Berenstein and Kirillov \cite{BK95} studied relations satisfied by the involutions $t_i$, introducing the Berenstein--Kirillov group $\mathcal{BK}$ (or Gelfand--Tsetlin group), the free group generated by $t_i$, modulo the relations they satisfy on semistandard Young tableaux of any shape 
\cite{BK16,CGP16}. Chmutov, Glick and Pylyavskyy 
\cite{CGP16}, using semistandard growth diagrams, found precise implications between sets of relations in the cactus group $J_n$ 
\cite{HenKam06} and the Berenstein–Kirillov group $\mathcal{BK}_n$, the subgroup of $\mathcal{BK}$ generated by $t_1, \ldots, t_{n-1}$, concluding that $\mathcal{BK}_n$ is isomorphic to a quotient of the cactus group $J_n$, and yielding a presentation for the cactus group in terms of Bender-Knuth generators. Halacheva has remarked 
\cite[Remark 3.9]{Hala20} that this isomorphism may also be obtained by noting the coincidence of the actions of both groups on a crystal of semistandard Young tableaux of straight shape, filled in $[n] := \{1 \!<\! \ldots \!<\! n\}$ 
\cite{Hala16,Hala20}.

Bender--Knuth involutions have been defined by Stembridge for shifted tableaux in 
\cite[Section 6]{Stem90}, but they are not compatible with the canonical form of those tableaux (Section \ref{sec:background}). Motivated by the coincidence of the tableau switching of two adjacent letters, on type $A$ tableaux, with the classic Bender--Knuth involutions, we introduce a shifted version of the Bender--Knuth operators, here denoted $\mathsf{t}_i$, for shifted semistandard tableaux, using the shifted tableau switching introduced by Choi, Nam, and Oh 
\cite{CNO17}. Alternatively, we may use type $C$ infusion 
\cite{TY09} together with the semistandardization \cite{PY17}. Using the shifted Bender--Knuth involutions, we define a shifted version of the Berenstein--Kirillov group, denoted $\mathcal{SBK}$, with $\mathcal{SBK}_n$ being defined analogously. The shifted Bender--Knuth involutions satisfy in $\mathcal{SBK}$ all the relations that the Bender--Knuth involutions are known to satisfy in $\mathcal{BK}$, except $(\mathsf{t}_1 \mathsf{t}_2)^6 = 1$, which does not need to hold (see \cref{rmk:t1t26}). This has no effect on the results that follow, as this relation does not follow from the cactus group relations, similarly to the case with the classic Bender--Knuth involutions 
\cite[Remark 1.9]{CGP16}. 

Following the work in \cite{CGP16, Hala16, Hala20}, we show that the action of $\mathcal{SBK}_n$ on the shifted tableau crystal of Gillespie, Levinson and Purbhoo 
\cite{GLP17} $\mathsf{ShST}(\lambda,n)$, the crystal-like structure on shifted semistandard tableaux of straight shape $\lambda$ and filled with $[n]' := \{1'\!<\!1<\! \cdots \!<\!n' \!<\! n\}$, coincides with the one of $J_n$ on the same crystal 
\cite{Ro20b}, concluding that $\mathcal{SBK}_n$ is isomorphic to a quotient of the cactus group, and due to 
\cite[Theorem 1.8]{CGP16} we have in \eqref{eq:cact_prsnt} another presentation of the cactus group via the shifted Bender--Knuth involutions. 

This paper is organized as follows. \cref{sec:background} provides the basic definitions and algorithms on shifted tableaux, in particular, the reversal and evacuation, as well as the main concepts regarding the shifted tableau switching 
\cite{CNO17}. In \cref{sec:crystal} we briefly recall the basic structure of the shifted tableau crystal 
\cite{GLP17}, and an action of the cactus group \cite{Ro20b}. In \cref{secBK}, we introduce the shifted Bender--Knuth operators $\mathsf{t}_i$ (\cref{def:sbk_mov}), using the shifted tableau switching, and then define a shifted Berenstein--Kirillov group. We then prove the main result (\cref{teo:cact_sbk}) stating that the shifted Berenstein--Kirillov group is isomorphic to a quotient of the cactus group.

This is an extended abstract of a full paper to appear.

\section{Background}\label{sec:background}
A \emph{strict partition} is a sequence $\lambda = (\lambda_1 \!>\! \cdots \!>\! \lambda_k)$ of positive integers, called the parts of $\lambda$, displayed in decreasing order. 
A strict partition $\lambda$ is identified with its \emph{shifted shape} $S(\lambda)$, a diagram whose $i$-th row have $\lambda_i$ boxes, with each row being shifted $i-1$ units to the right. 
Skew shapes are defined as expected, with shapes of the form $\lambda/\emptyset$ being called \emph{straight}.
A shifted shape $\lambda$ lies naturally in the ambient triangle of the \emph{shifted staircase shape} $\delta = (\lambda_1, \lambda_1-1, \ldots, 1)$. We define the \emph{complement} of $\lambda$ to be the strict partition $\lambda^{\vee}$ whose set of parts is the complement of the set of parts of $\lambda$ in $\{\lambda_1, \lambda_1-1, \ldots, 1\}$. 
In particular, $\emptyset^{\vee} = \delta$. We consider the \emph{primed} alphabet  $[n]' := \{ 1'\!<\! 1 \!<\! \cdots \!<\! n' \!<\! n\}$.
	
Given strict partitions $\lambda$ and $\mu$ such that $\mu \subseteq \lambda$, a \emph{shifted semistandard tableau} $T$ of shape $\lambda / \mu$ is a filling of $\lambda/\mu$ with letters in $[n]'$ such that the entries are weakly increasing in each row and in each column, and there is at most one $i$ per column and one $i'$ per row, for any $i \geq 1$. 
The \emph{reading word} $w(T)$ of a shifted tableau is obtained by reading its entries from left to right, going bottom to top. 
The \emph{weight} of $T$ is defined as the weight of its word. A word or a shifted tableau are said to be \emph{standard} if their weight is $(1, \ldots, 1)$. 
Words and tableaux will be presented in \emph{canonical form}, i.e., the first occurrence of each letter $i$ or $i'$ must be unprimed 
\cite[Definition 2.1]{GLP17}. The set of shifted semistandard tableaux of shape $\lambda/\mu$, on the alphabet $[n]'$, in canonical form, is denoted by $\mathsf{ShST}(\lambda/\mu,n)$. 
For instance, the following is a shifted semistandard tableau of shape $(6,3,1)/(3,1)$, with its word and weight:
$$T=\begin{ytableau}
{} & {} & {} & 1 & 1 & 2'\\
\none & {} & 2 & 2'\\
\none & \none & 3
\end{ytableau}\qquad w(T)=322'112' \qquad \mathsf{wt}(T) = (2,3,1).$$

\subsection{Shifted evacuation and reversal}\label{ss:evac}
The shifted \emph{jeu de taquin} is defined similarly to the one for ordinary Young tableaux, with an exception for certain slides on the main diagonal (see 
\cite{Wor84}). The \emph{rectification} $\mathsf{rect}(T)$ of $T$ is the tableau obtained by applying any sequence of inner slides until a straight shape is obtained (it does not depend on the chosen sequence of slides). Two tableaux are said to be \emph{shifted Knuth equivalent} if they have the same rectification 
\cite[Theorem 6.4.17]{Wor84}. An operator on shifted tableaux that commutes with the shifted \textit{jeu de taquin} is called \emph{coplactic}. Two shifted semistandard tableaux are \emph{shifted dual equivalent} (or coplactic equivalent) if they have the same shape after applying any sequence (including the empty one) of shifted \textit{jeu de taquin} slides to both. 

Given $T\in \mathsf{SShT}(\lambda/\mu, n)$, its \emph{complement} in $[n]'$ is the tableau $\mathsf{c}_n (T)$ obtained by reflecting $T$ along the anti-diagonal in its shifted staircase shape, while complementing the entries by $i \mapsto (n-i+1)'$ and $i' \mapsto n-i+1$. Note that if $T$ is of shape $\lambda/\mu$, then $\mathsf{c}_n (T)$ is of shape $ \mu^{\vee} / \lambda^{\vee}$, and if $\mathsf{wt}(T) = (wt_1, \ldots, wt_n)$, then $\mathsf{wt}(\mathsf{c}_n(T)) = \mathsf{wt}(T)^{\mathsf{rev}} := (wt_n, \ldots, wt_1)$. Haiman 
\cite[Theorem 2.13]{Haim92} showed that, given $T \in \mathsf{ShST}(\lambda/\mu,n)$, there exists a unique tableau $T^e$, the \emph{reversal} of $T$, that is shifted Knuth equivalent to $\mathsf{c}_n (T)$ and dual equivalent to $T$. If $T$ is straight-shaped, $T^e$ is known as the \emph{evacuation} of $T$ and denoted $\mathsf{evac}(T) = \mathsf{rect} (\mathsf{c}_n (T))$. Since the operator $\mathsf{c}_n$ preserves shifted Knuth equivalence 
\cite[Lemma 7.1.4]{Wor84}, the reversal operator is the coplactic extension of evacuation, in the sense that, we may first rectify $T$, then apply the evacuation operator, and then perform outer \textit{jeu de taquin} slides, in the reverse order defined by the previous rectification, to get a tableau $T^e$ with the same shape of $T$. From 
\cite[Corollaries 2.5, 2.8 and 2.9]{Haim92}, this tableau $T^e$ is shifted dual equivalent to $T$, besides being shifted Knuth equivalent to $\mathsf{c}_n (T)$.

\subsection{Shifted tableau switching}\label{sec:switching}
In this section we recall the shifted tableau switching algorithm 
\cite{CNO17}, which will be used later in \cref{secBK} to introduce a shifted version of the Bender--Knuth involutions. Unlike the tableau switching algorithm for type $A$ tableaux 
\cite{BSS96}, the shifted version depends on the order in which the switches are performed, similarly to the infusion map 
\cite{TY09}. As remarked in \cite[Remark 8.1]{CNO17}, the output of this algorithm can be recovered by applying the semistandardization 
\cite{PY17} on the type $C$ infusion map of standardized tableaux 
\cite{TY09}, since the latter coincide with the shifted tableau switching on standard tableaux. 

We begin with the definitions of the shifted tableau switching for pairs $(A,B)$ of border strip shifted tableaux, with $B$ extending $A$, and for pairs of shifted semistandard tableaux $(S,T)$, with $T$ extending $S$. We omit most of the details and proofs, and refer to \cite{CNO17}. We write $\mathbf{i}$ when referring to the letters $i$ and $i'$ without specifying whether they are primed. Given $T \in \mathsf{ShST}(\lambda/\mu,n)$, we denote by $T^i$ the border strip obtained from $T$ considering only the letters $\{i',i\}$.
	
Given $S(\lambda/\mu)$ a double border strip, i.e., a shape not containing a subset of the form $\{(i,j),(i+1,j+1),(i+2,j+2)\}$, a \emph{shifted perforated $\mathbf{a}$-tableau} is a filling of some of the boxes of $S(\lambda/\mu)$ with letters $a$, $a'$ $\in [n]'$ such that no $a'$-boxes are south-east to any $a$-boxes, there is at most one $a$ per column and one $a'$ per row, and the main diagonal has at most one $\mathbf{a}$. Given a perforated $\mathbf{a}$-tableau $A$ and a perforated $\mathbf{b}$-tableau $B$, the pair $(A,B)$ is said to be a \emph{shifted perforated $(\mathbf{a},\mathbf{b})$-pair} of shape $\lambda/\mu$, for $S(\lambda/\mu)$ a double border strip, if $sh(A) \sqcup sh(B) = S(\lambda/\mu)$. If $A$ and $B$ are perforated tableaux, we say that $B$ \emph{extends} $A$ if $sh(B)$ extends $sh(A)$, denoting by $A\sqcup B$  the filling obtained by putting $A$ and $B$ next to each other. If $(A,B)$ is a shifted perforated $(\mathbf{a,b})$-pair, one can interchange an $\mathbf{a}$-box with a $\mathbf{b}$-box in $A \sqcup B$ subject to the moves depicted in \cref{fig:switches}, called the \emph{(shifted) switches}. A $\mathbf{a}$-box is said to be \emph{fully switched} if it can't be switched with any $\mathbf{b}$-boxes, and that $A \sqcup B$ if \emph{fully switched} if every $\mathbf{a}$-box is fully switched.

\begin{figure}[h]
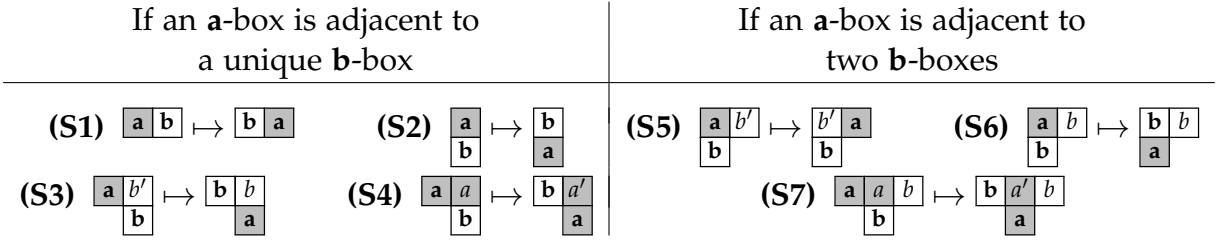

\begin{center}
\begin{tabular}{c|c}
If an $\mathbf{a}$-box is adjacent to & If an $\mathbf{a}$-box is adjacent to\\
a unique $\mathbf{b}$-box  &  two $\mathbf{b}$-boxes  \\
\hline
\\[-0.5em]
\textbf{(S1)}\; $
	\ytableausetup{smalltableaux}	
	\begin{ytableau}
    *(lightgray)\mathbf{a} & \mathbf{b}
  \end{ytableau} \mapsto
  \begin{ytableau}
  \mathbf{b} & *(lightgray)\mathbf{a}  \end{ytableau}$
  \hspace{2em}
  \textbf{(S2)}\; $\begin{ytableau}
    *(lightgray)\mathbf{a}\\
    \mathbf{b}
  \end{ytableau} \mapsto
  \begin{ytableau}
  \mathbf{b}\\
  *(lightgray)\mathbf{a}  \end{ytableau}$ & \textbf{(S5)}\; $\begin{ytableau}
    *(lightgray)\mathbf{a} & b'\\
    \mathbf{b}
  \end{ytableau} \mapsto
  \begin{ytableau}
  b' & *(lightgray)\mathbf{a} \\ \mathbf{b}  \end{ytableau}$ 
  \hspace{2em}
  \textbf{(S6)}\; $\begin{ytableau}
    *(lightgray)\mathbf{a} & b\\
    \mathbf{b}
  \end{ytableau} \mapsto
  \begin{ytableau}
  \mathbf{b} &  b \\*(lightgray)\mathbf{a}  \end{ytableau}$\\
  \\[-1em]

\textbf{(S3)}\; $\begin{ytableau}
    *(lightgray)\mathbf{a} & {b'}\\
    \none & \mathbf{b}
  \end{ytableau} \mapsto
  \begin{ytableau}
  \mathbf{b} &  b \\ \none & *(lightgray)\mathbf{a}  \end{ytableau}$
 \hspace{2em}
  \textbf{(S4)} \;$\begin{ytableau}
    *(lightgray)\mathbf{a} & *(lightgray)a\\
    \none & \mathbf{b}
  \end{ytableau} \mapsto
  \begin{ytableau}
  \mathbf{b} & *(lightgray){a'} \\
  \none & *(lightgray)\mathbf{a}  \end{ytableau}$ &
    \textbf{(S7)} \;$\begin{ytableau}
    *(lightgray)\mathbf{a} & *(lightgray){a} & b\\
    \none & \mathbf{b}
  \end{ytableau} \mapsto
  \begin{ytableau}
  \mathbf{b} & *(lightgray){a'} & b\\
  \none & *(lightgray)\mathbf{a}  \end{ytableau}$ 
\\
\end{tabular}
\end{center}
\caption{The shifted switches \cite[Section 3]{CNO17}.}
\label{fig:switches}
\end{figure}

\begin{defin}[\cite{CNO17}]
Let $T = A \sqcup B$ be a perforated $\mathbf{(a,b)}$-pair not fully switched. The \emph{shifted switching process} from $T$ to $\varsigma^m(T)$, with $m$ the least integer such that $\varsigma^m(T)$ is fully switched, is obtained as follows: choose the rightmost $a$-box in $A$ that is adjacent on the north or west to a $\mathbf{b}$-box, if it exists, otherwise, choose the bottommost $a'$-box in the same conditions, and then apply the adequate switch, obtaining $\varsigma(T)$. The process is repeated until $\varsigma^m (T)$ is fully switched, and in this case, we denote $A_B := (\varsigma^m (T))^a$ and $^A B := (\varsigma^m (T))^b$, the perforated tableaux obtained from $\varsigma^m (T)$ considering only the letters $\{a',a\}$ and $\{b',b\}$ respectively.
\end{defin}

This process is well defined and it is an involution \cite[Theorem 3.5]{CNO17}. It may be extended to pairs of shifted semistandard tableaux $(S,T)$, with $T$ extending $S$, by applying the shifted switching process sequentially to the shifted pairs $(S^m,T^1)$, $(S^m, T^2)$, $\ldots$, $(S^m,T^n)$,$\ldots$, $(S^1, T^1)$, $(S^1,T^2)$,$\ldots$, $(S^1,T^1)$, where $m$ and $n$ are the maximum entries of $S$ and $T$. This process on pairs of shifted semistandard tableaux is also well defined \cite[Theorem 3.6]{CNO17} and it is an involution \cite[Theorem 4.3]{CNO17}. Moreover, it is compatible with standardization \cite[Remark 3.8]{CNO17} and with canonical form. 

\begin{ex}
The following illustrates the shifted tableau switching on a pair of tableaux:

$$(S,T) = \begin{ytableau}
*(lgray) 1 & *(lgray) 1 & *(lgray) 1 & 1\\
\none & *(lgray) 2 & 1
\end{ytableau} \xrightarrow{\textbf{(S1)}}
\begin{ytableau}
*(lgray) 1 & *(lgray) 1 & *(lgray) 1 & 1\\
\none & 1 & *(lgray) 2
\end{ytableau} \xrightarrow{\textbf{(S1)}}
\begin{ytableau}
*(lgray) 1 & *(lgray) 1 & 1 & *(lgray) 1\\
\none & 1 & *(lgray) 2
\end{ytableau} \xrightarrow{\textbf{(S7)}}
\begin{ytableau}
1 & *(lgray) 1' & 1 & *(lgray) 1\\
\none & *(lgray) 1 & *(lgray) 2
\end{ytableau} \xrightarrow{\textbf{(S1)}}
\begin{ytableau}
1 & 1 & *(lgray) 1' & *(lgray) 1\\
\none & *(lgray) 1 & *(lgray) 2
\end{ytableau} = (^S T,S_T).
$$
\end{ex}

Another algorithm for tableaux of straight shape, that coincides with the shifted evacuation (\cref{ss:evac}), using the shifted tableau switching, is presented in \cite{CNO17}. Using an auxiliary alphabet, it applies the shifted switching process sequentially to the pairs $(T^1,T^2 \sqcup \cdots \sqcup T^n)$, $(T^2, T^3 \sqcup \cdots \sqcup T^n)$,$\ldots$, $(T^{n-1},T^n)$, for $T \in \mathsf{ShST}(\nu,n)$. This algorithm coincides with the shifted evacuation for straight-shaped tableaux \cite[Theorem 5.6]{CNO17}. It also may be modified to obtain a restriction $\mathsf{evac}_k$ to the alphabet $\{1, \ldots, k\}'$, for $k \leq n$, by applying $\mathsf{evac}$ to $T^1 \sqcup \cdots \sqcup T^k$ and maintaining $T^{k+1} \sqcup \cdots \sqcup T^n$ unchanged. Similarly to the ordinary Young tableaux case \cite[Section 5]{BSS96}, the shifted evacuation algorithms, may be extended to skew shapes, by performing the adequate shifted switching algorithms on a given skew shape. We denote these operators by $\widetilde{\mathsf{evac}}$ and $\widetilde{\mathsf{evac}}_k$. Like in type $A$, the operator $\widetilde{\mathsf{evac}}$ is different from the reversal (\cref{ss:evac}), as in general, given $T \in \mathsf{ShST}(\lambda/\mu,n)$, $\widetilde{\mathsf{evac}} (T)$ does not need to be shifted Knuth equivalent to $\mathsf{c}_n (T)$.

\section{Shifted tableau crystal and a cactus group action}\label{sec:crystal}
Gillespie, Levinson and Purbhoo \cite{GLP17} introduced a crystal-like structure on $\mathsf{ShST}(\lambda/\mu,n)$. The author shows in \cite{Ro20b} that the cactus group acts naturally on this structure via the restrictions of the shifted Schützenberger involution to primed subintervals of $[n]$.

The shifted tableau crystal consists on a crystal-like structure on $\mathsf{ShST}(\lambda/\mu,n)$, together with primed and unprimed raising and lowering operators $E_i$, $E_i'$, $F_i$ and $F_i'$, lenght functions $\varphi_i$ and $\varepsilon_i$, for each $i \in I := [n-1]$, and a weight function. For the sake of brevity, we omit these definitions and refer to the original work in \cite{GL19,GLP17}. We use the notation $\mathsf{ShST}(\lambda/\mu,n)$ for both the set and its crystal-like structure.
It may be regarded as a directed acyclic graph with weighted vertices, and $i$-coloured labelled double edges, solid ones for unprimed operators, and dashed ones for primed operators (see \cref{fig:crystal_t2}). This graph is partitioned into $i$-\emph{strings}, which are the $\{i',i\}$-connected components of $\mathsf{ShST}(\lambda/\mu,n)$, for each $i \in I$. There are two possible arrangements for these strings \cite[Section 3.1]{GL19} \cite[Section 8]{GLP17}: \emph{separated strings}, consisting of two $i$-labelled chains of equal length, connected by $i'$-labelled edges, and \emph{collapsed strings} a double chain of both $i$- and $i'$-labelled edges. Additionally, $\mathsf{ShST}(\lambda/\mu,n)$ decomposes into connected components, each one having a unique highest weight element (an element for which all primed and unprimed raising operators are undefined) corresponding to a LRS tableau, and a unique lowest weight element (defined analogously with lowering operators), the reversal of it. Thus, each of these connected components is isomorphic, via rectification, to $\mathsf{ShST}(\nu,n)$, for some strict partition $\nu$ \cite[Corollary 6.5]{GLP17}.

\subsection{The Schützenberger involution and the crystal reflection operators}
The Schützenberger or Lusztig involution is defined on the shifted tableau crystal \cite[Section 2.3.1]{GL19} in the same fashion as for type $A$ Young tableau crystal. It is realized by the shifted evacuation (for straight shapes) or the shifted reversal (for skew shapes). The shifted crystal reflection operators $\sigma_i$, for $i \in I$, were introduced in \cite{Ro20b}, using the crystal operators. They coincide with the restriction of the Schützenberger involution to the intervals of the form $\{i,i+1\}'$. 

\begin{prop}[\cite{Ro20b}]\label{defSchu}
There exists a unique map of sets $\eta: \mathsf{ShST}(\nu,n) \longrightarrow \mathsf{ShST}(\nu,n)$ that satisfies the following, for all $T\in \mathsf{ShST}(\nu,n)$ and for all $i \in I$:

1. $E'_i \eta (T) = \eta F'_{n-i} (T)$ and $E_i \eta (T) = \eta F_{n-i} (T)$.

2. $F'_i \eta (T) = \eta E'_{n-i} (T)$ and $F_i \eta (T) = \eta E_{n-i} (T)$.

3. $wt(\eta(T)) = wt(T)^{\mathsf{rev}}$.
\end{prop}

This map, called the \emph{Schützenberger or Lusztig involution}, is defined on $\mathsf{ShST}(\lambda/\mu,n)$ by extending it to its connected components. It coincides with the evacuation $\mathsf{evac}$ in $\mathsf{ShST}(\nu,n)$, and with the reversal $e$ on the connected components of $\mathsf{ShST}(\lambda/\mu,n)$. The map $\eta$ is a coplactic and a weight-reversing, shape-preserving involution. 

Let $[i,j]:=\{i\!<\! \cdots \!<\! j\}$, for $1 \leq i < j \leq n$, and let $\theta_{i,j}$ denote the longest permutation in $\mathfrak{S}_{[i,j]}$ embedded in $\mathfrak{S}_{n}$. Given $T \in \mathsf{ShST}(\lambda/\mu,n)$ and $1 \leq i < j \leq n$, let $T^{i,j} := T^i \sqcup \cdots \sqcup T^j$. We define the \emph{restriction of the Schützenberger involution} to the interval $[i,j]'$ as $\eta_{i,j} (T) := T^{1,i - 1} \sqcup \eta(T^{i,j}) \sqcup T^{j+1, n}$. In particular, we have $\eta_{1,n} = \eta$ and $\eta_{i,i+1}=\sigma_i$. The \emph{shifted crystal reflection operators} $\sigma_i$, for $i \in I$, introduced in \cite[Definition 4.3]{Ro20b} using the shifted tableau crystal operators, are involutions which coincide with $\eta_{i,i+1}$, the restriction of the Schützenberger involution to the intervals of the form $\{i,i+1\}'$. Unlike the type $A$ case, the reflection operators $\sigma_i$ do not define an action of the symmetric group $\mathfrak{S}_n$ on $\mathsf{ShST}(\lambda/\mu,n)$ as the braid relations $(\sigma_i \sigma_{i+1})^3=1$ do not need not hold.

\subsection{An action of the cactus group}\label{sec:cactusaction}
Halacheva \cite{Hala16} has shown that there is a natural action of the cactus group $J_{\mathfrak{g}}$ on any $\mathfrak{g}$-crystal, for $\mathfrak{g}$ a complex, reductive, finite-dimensional Lie algebra. In particular, the cactus group $J_n$ (corresponding to $\mathfrak{g}=\mathfrak{gl}_n$) acts internally on the type $A$ crystal of semistandard Young tableux, via the partial Schützenberger involutions, or partial evacuations for straight shapes. Following a similar approach, the author has shown in 
\cite{Ro20b} that there is a natural action of the cactus group $J_n$ on $\mathsf{ShST}(\lambda/\mu,n)$. This action is realized by the restricted shifted Scützenberger involutions $\eta_{i,j}$. 

\begin{defin}[\cite{HenKam06}]\label{def:cactus}
The \emph{$n$-fruit cactus group} $J_n$ is the free group with generators $s_{i,j}$, for $1 \leq i < j \leq n$, subject to the relations:
\begin{equation}\label{eq:cactus_rel}
s_{i,j}^2 = 1, \quad s_{i,j} s_{k,l} = s_{k,l} s_{i,j}\; \text{for} \; [i,j] \cap [k,l] = \emptyset, \quad s_{i,j}s_{k,l} = s_{i+j-l,i+j-k} s_{i,j}\; \text{for} \; [k,l] \subseteq [i,j].
\end{equation}
\end{defin}

There is an epimorphism $J_n \longrightarrow \mathfrak{S}_n$, sending $s_{i,j}$ to $\theta_{i,j}$. The kernel of this surjection is known as the \emph{pure cactus group}  (see \cite[Section 3.4]{HenKam06}). The first and third relations ensure that we may only consider generators of the form $s_{1,k}$, since any $s_{i,j}$ may be written as
\begin{equation}\label{eq:cactus_s1k}
s_{i,j} = s_{1,j} s_{1,j-i+1} s_{1,j}.
\end{equation}

\begin{teo}[{\cite[Theorem 5.7]{Ro20b}}]\label{thm:cactusaction}
There is a natural action of the $n$-fruit cactus group $J_n$ on the shifted tableau crystal $\mathsf{ShST}(\lambda/\mu,n)$ given by the group homomorphism $\phi : s_{i,j} \mapsto \eta_{i,j}$, for $1 \leq i < j \leq n$.
\end{teo}

Recall that $\mathsf{evac}_j (T) = \mathsf{evac} (T^{1,j}) \sqcup T^{j+1,n}$, for $T \in \mathsf{ShST}(\nu,n)$.  As a consequence of $\phi$ being an homomorphism, we have the next result.

\begin{cor}[{\cite[Corollary 5.8]{Ro20b}}]\label{cor:eta_cactus}
Let $T \in \mathsf{ShST}(\lambda/\mu,n)$ and $1 \leq i < j \leq n$. Then, $\eta_{i,j} (T) = \eta_{1,j} \eta_{1,j-i+1} \eta_{1,j} (T)$. In particular, for $T$ a straight-shaped tableau, we have 
$\eta_{i,j}(T) = \mathsf{evac}_{j} \mathsf{evac}_{j-i+1} \mathsf{evac}_{j} (T).$
\end{cor}
The identities \eqref{eq:cactus_s1k} and $\eta_{1,i}=\mathsf{evac}_i$ on straight-shaped tableaux incur the next result.

\begin{cor}\label{cor:cact_evaci}
There is a natural action of the $n$-fruit cactus group on the shifted tableau crystal $\mathsf{ShST}(\nu,n)$, given by the group homomorphism $s_{1,i} \mapsto \mathsf{evac}_i$, for $i \in I$.
\end{cor}

\section{A shifted Berenstein--Kirillov group}\label{secBK}
We use the shifted Bender--Knuth operators to introduce a shifted Berenstein--Kirillov group. Following the works of Halacheva 
\cite{Hala16,Hala20} and Chmutov, Glick and Pylyavskyy \cite{CGP16}, we then show that this group is isomorphic to a quotient of the cactus group and exhibit an alternative presentation for the cactus group. 

\subsection{Shifted Bender--Knuth involutions}	
We now introduce the \emph{shifted Bender--Knuth moves} $\mathsf{t}_i$. 
These operators differ from the ones introduced by Stembridge \cite[Section 6]{Stem90}, which are not compatible with canonical form.
We first fix some notation. Given $i \in I = [n-1]$, recall that $\theta_i \in \mathfrak{S}_n$ denotes the simple transposition $(i,i+1)$. We extend   $\theta_i$ to $[n]'$ by putting $\theta_i( x') := (\theta_i (x))'$, for $x\in [n]$. 

\begin{defin}\label{def:sbk_mov}
Given $T \in \mathsf{ShST}(\lambda/\mu,n)$ and $i \in I$, we define the \emph{shifted Bender--Knuth move} $\mathsf{t}_i$ as  
$\mathsf{t}_i (T) := \theta_i \circ \mathsf{SP}_{i,i+1} (T),$
where $\mathsf{SP}_{i,i+1} (T)$ is obtained from $T$ by applying the shifted tableau switching to the pair $(T^i,T^{i+1})$, leaving the remaining letters unchanged.
\end{defin}

\begin{ex}
Let $T = \begin{ytableau}
1 & 1 & 2' & 2\\
\none & 2 & 3'\\
\none & \none & 3
\end{ytableau}$. Then, we have:

$$\begin{ytableau}
*(lgray)1 & *(lgray)1 & 2' & 2\\
\none & 2 & {\color{gray} 3'}\\
\none & \none & {\color{gray} 3}
\end{ytableau} \xrightarrow{\textbf{(S5)}}
\begin{ytableau}
*(lgray)1 & 2' & *(lgray)1 & 2\\
\none & 2 & {\color{gray} 3'}\\
\none & \none & {\color{gray} 3}
\end{ytableau} \xrightarrow{\textbf{(S1)}}
\begin{ytableau}
*(lgray)1 & 2' & 2 & *(lgray)1\\
\none & 2 & {\color{gray} 3'}\\
\none & \none & {\color{gray} 3}
\end{ytableau} \xrightarrow{\textbf{(S3)}}
\begin{ytableau}
2 & 2 & 2 & *(lgray)1\\
\none & *(lgray)1 & {\color{gray} 3'}\\
\none & \none & {\color{gray} 3}
\end{ytableau} \xrightarrow{\theta_1}
\begin{ytableau}
1 & 1 & 1 & 2\\
\none & 2 & 3'\\
\none & \none & 3
\end{ytableau} = \mathsf{t}_1 (T).
$$

$$\begin{ytableau}
{\color{gray} 1} & {\color{gray} 1} & *(lgray)2' & *(lgray)2\\
\none & *(lgray)2 & 3'\\
\none & \none & 3
\end{ytableau} \xrightarrow{\textbf{(S3)}}
\begin{ytableau}
{\color{gray} 1} & {\color{gray} 1} & *(lgray)2' & *(lgray)2\\
\none & 3 & 3\\
\none & \none & *(lgray) 2
\end{ytableau} \xrightarrow{\textbf{(S2)}}
\begin{ytableau}
{\color{gray} 1} & {\color{gray} 1} & 3 & *(lgray)2\\
\none & 3 & *(lgray)2'\\
\none & \none & *(lgray) 2
\end{ytableau} \xrightarrow{\theta_2}
\begin{ytableau}
1 & 1 & 2 & 3\\
\none & 2 & 3'\\
\none & \none & 3
\end{ytableau} = \mathsf{t}_2(T).
$$
\end{ex}

\begin{figure}[h]
\begin{center}
\includegraphics[scale=0.52]{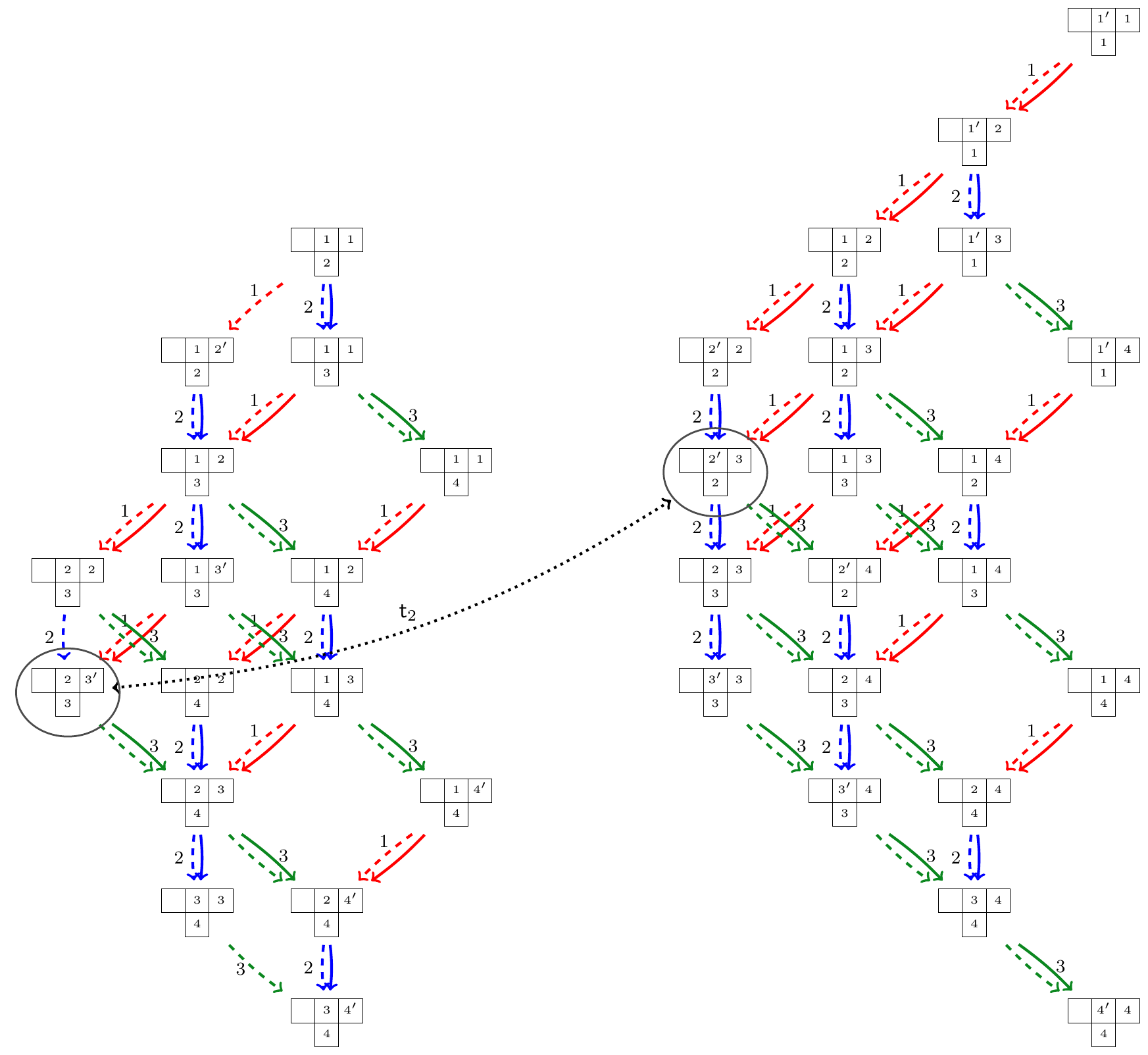}
\end{center}
\caption{An example of the action of $\mathsf{t}_2$ on the shifted tableau crystal graph $\mathsf{ShST}((3,1)/(1),4)$, which has two connected components.}
\label{fig:crystal_t2}
\end{figure}

The operators $\mathsf{t}_i$ satisfy $\mathsf{t}_i^2=1$ and $\mathsf{t}_i \mathsf{t}_j = \mathsf{t}_j \mathsf{t}_i$, for $|i-j|>1$, and they act via $\theta_i$ on the weight of a tableau. 
Like the case in type $A$, they are not coplactic, and, in general, $\mathsf{t}_i \neq \sigma_i$ (although $\mathsf{t}_1$ and $\sigma_1$ coincide on straight-shaped tableaux). Moreover, $\mathsf{t}_i (T)$ does not need to be in the same $i$-string as $T$ (see \cref{fig:crystal_t2}). We define the \emph{shifted promotion operator} $\mathsf{p}_i$ as $\mathsf{p}_i (T) := \mathsf{t}_i \mathsf{t}_{i-1} \cdots \mathsf{t}_1 (T)$, for $T \in \mathsf{ShST}(\lambda/\mu,n)$ and $i \in I$.

\begin{prop}\label{evacpi}
Given $T \in \mathsf{ShST}(\lambda/\mu,n)$ and $i \in I$, we have $\widetilde{\mathsf{evac}}_{i+1} (T) = \mathsf{p}_1 \mathsf{p}_2 \cdots \mathsf{p}_i (T).$ In particular, for $T \in \mathsf{ShST}(\nu,n)$ we have $\eta_{1,i+1}(T)=\mathsf{evac}_{i+1} (T) = \mathsf{p}_1 \mathsf{p}_2 \cdots \mathsf{p}_i (T)$.
\end{prop}

Similar to $\eta_{i,j}$, we may define the restriction of $\widetilde{\mathsf{evac}}$ to the interval $[i,j]'$ by $\widetilde{\mathsf{evac}}_{i,j}(T):= T^{1,i-1} \sqcup \widetilde{\mathsf{evac}}(T^{i,j}) \sqcup  T^{j+1,n}$. However, these operators do not need to satisfy the relation $\widetilde{\mathsf{evac}}_{i,j}=\widetilde{\mathsf{evac}}_{j} \widetilde{\mathsf{evac}}_{j-i+1} \widetilde{\mathsf{evac}}_{j}$, unlike the operators $\eta_{i,j}$ (see \cref{cor:eta_cactus}).

\subsection{The Berenstein--Kirillov group}	

The Bender--Knuth moves $t_i$, for $i \in I$, are involutions on semistandard Young tableaux filled in $[n]$, that act only on the letters $\{i,i+1\}$, reverting their weight 
\cite{BeKn72}. They are known to coincide with the tableau switching on type $A$ on two consecutive letters 
\cite{BSS96}. The \emph{Berenstein--Kirillov group} $\mathcal{BK}$ (or \emph{Gelfand--Tsetlin group}), is the free group generated by these involutions $t_i$, for $i > 0$, modulo the relations they satisfy on semistandard Young tableaux of any shape 
\cite{BK16, BK95, CGP16}. The following are some of the relations known to hold in $\mathcal{BK}$ \cite[Corollary 1.1]{BK95}
$$t_i^2 = 1, \qquad t_i t_j = t_j t_i,\; \text{for}\; |i-j|>1,\qquad (t_1 t_2)^6 = 1, \qquad (t_1 q_i)^4 = 1,\; \text{for}\; i > 2,$$
where $q_{i} := t_1 (t_2 t_1) \cdots (t_i t_{i-1} \cdots t_1)$, for $i \geq 1$, are involutions. Let $q_{k,j} := q_{j-1} q_{j-k} q_{j-1}$, for $k<j$. In particular, $q_i = q_{1,i+1}$. Another relation was found in 
\cite[Theorem 1.6]{CGP16}, which generalizes the last one:
$$(t_i q_{j,k})^2 = 1, \;\text{for}\; i+1<j<k.$$
Let $\mathcal{BK}_n$ be the subgroup of $\mathcal{BK}$ generated by $t_1, \ldots, t_{n-1}$. The involutions $q_i$, for $i \in I$, provide another set of generators, and their action on straight-shaped Young tableaux coincide with the one of the restriction of the Schützenberger involution to $[i+1]$ \cite[Remark 1.3]{BK95}. It was shown in 
\cite{CGP16}, using semistandard growth diagrams, that $\mathcal{BK}_n$ is isomorphic to a quotient of the cactus group. This result could also be derived by noting the coincidence of the actions of $J_n$ 
\cite{Hala16} and $\mathcal{BK}_n$ on straight-shaped semistandard Young tableaux as noted in \cite[Remark 3.9]{Hala20}.

\begin{teo}
The group $\mathcal{BK}_n$ is isomorphic to a quotient of $J_n$, as a result of the following being group epimorphisms from $J_n$ to $\mathcal{BK}_n$:

1. $s_{i,j} \mapsto q_{i,j}$ \cite[Theorem 1.4]{CGP16}.

2. $s_{1,j} \mapsto q_{j-1}$ \cite[Remark 1.3]{BK95}, \cite[Section 10.2]{Hala16}, \cite[Remark 3.9]{Hala20}.
\end{teo}

Chmutov, Glick and Pylyavskyy \cite{CGP16} established an equivalence between relations that are satisfied in $\mathcal{BK}_n$ and the ones of the cactus group $J_n$ \eqref{eq:cactus_rel}, thus obtaining an alternative presentation for the latter via the Bender--Knuth involutions. 

\begin{teo}[{\cite[Theorem 1.8]{CGP16}}]\label{thm:rel_cact_bk}
The relations 
$$
t_i^2 = 1, \qquad t_i t_j = t_j t_i, \;\text{for}\; |i-j| > 1, \qquad (t_i q_{k-1} q_{k-j} q_{k-1})^2 = 1, \;\text{for}\; i+1<j<k,
$$
where $q_i := t_1 (t_2 t_1) \cdots (t_i t_{i-1} \cdots t_1)$, are equivalent to the cactus group relations \eqref{eq:cactus_rel} satisfied by the maps $q_{i,j}$
$$
q_{i,j}^2 = 1, \qquad q_{i,j}q_{k,l} = q_{i+j-l,i+j-k} q_{i,j}, \;\text{for}\; i \leq k < l \leq j, \qquad q_{i,j} q_{k,l} = q_{k,l} q_{i,j}, \;\text{for}\; j <k.
$$
\end{teo}

\subsection{A shifted Berenstein--Kirillov group and the cactus group}
Similar to the definition of the Berenstein--Kirillov group, we consider $\mathcal{SBK}$ to be the free group generated by the shifted Bender--Knuth involutions $\mathsf{t}_i$, for $i > 0$,  modulo the relations they satisfy when acting on shifted semistandard tableaux of any shape. We call it the \emph{shifted Berenstein--Kirillov group}, and consider its subgroup $\mathcal{SBK}_n$ generated by $\mathsf{t}_1, \ldots, \mathsf{t}_{n-1}$. We define the involutions $\mathsf{q}_i := \mathsf{t}_1 (\mathsf{t}_2 \mathsf{t}_1) \cdots (\mathsf{t}_i \mathsf{t}_{i-1} \cdots \mathsf{t}_1)$, for $i \in I$, which coincide with $\mathsf{evac}_{i+1}$ on straight-shaped shifted tableaux. We also set $\mathsf{q}_{i,j} := \mathsf{q}_{j-1}\mathsf{q}_{j-i}\mathsf{q}_{j-1}$, for $i<j$. In particular, $\mathsf{q}_{1,j}=\mathsf{q}_{j-1}$, for $j > 1$. We remark that, in general, $\widetilde{\mathsf{evac}}_{i,j} \neq \mathsf{q}_{i,j}$. However, \cref{cor:eta_cactus} ensures that $\mathsf{q}_{i,j} \in \mathcal{SBK}$ is realized by $\eta_{i,j}$ when acting on straight-shaped tableaux. 

\begin{prop}\label{prop:rels_SBK}
The following relations hold on $\mathcal{SBK}$:

1. $\mathsf{t}_i^2 = 1$, for $i>1$.

2. $\mathsf{t}_i \mathsf{t}_j = \mathsf{t}_j \mathsf{t}_i$, for $|i-j| >1$.

3. $(\mathsf{t}_i\mathsf{q}_{j,k})^2 =1$, for $2 \leq i+1 < j < k$. In particular, $(\mathsf{t}_1 \mathsf{}q_i)^4 = 1$, for $i > 2$.
\end{prop}

\begin{obs}\label{rmk:t1t26}
The operators $\mathsf{t}_i$ on shifted tableaux do not need to satisfy the relation $(\mathsf{t}_1 \mathsf{t}_2)^6 = 1$ in $\mathcal{SBK}$, as the next example shows: 
$$T = \begin{ytableau}
1 & 1 & 2' & 2 & 3\\
\none & 2 & 3' & 3\\
\none & \none & 3
\end{ytableau} \neq \begin{ytableau}
1 & 1 & 2' & 3'& 3\\
\none & 2 & 2 & 3\\
\none & \none & 3
\end{ytableau}=(\mathsf{t}_1 \mathsf{t}_2)^6 (T).$$
This has no effect in the next results, as the said relation does not follow from the cactus group relations \eqref{eq:cactus_rel}, similarly to the case with the classic Bender--Knuth involutions \cite[Remark 1.9]{CGP16}.
\end{obs}

\begin{lema}\label{lema:ti_si}
As elements of $\mathcal{SBK}$, we have
$$\mathsf{t}_1 = \mathsf{q}_1, \qquad \mathsf{t}_2 = \mathsf{q}_1 \mathsf{q}_2 \mathsf{q}_1, \qquad \mathsf{t}_i = \mathsf{q}_{i-1} \mathsf{q}_{i} \mathsf{q}_{i-1} \mathsf{q}_{i-2}, \,\text{for}\; i>2.$$
Consequently, $\mathsf{q}_1, \ldots, \mathsf{q}_{n-1}$ are generators for $\mathcal{SBK}_n$.
\end{lema}

\begin{teo}\label{teo:qi_evaci}
There is a natural action of $\mathcal{SBK}_n$ on $\mathsf{ShST}(\nu,n)$, given by the group homomorphism $\mathsf{q}_i \mapsto \mathsf{evac}_{i+1}$, which coincides with the action of $J_n$ as defined in \cref{cor:cact_evaci}.
\end{teo}

\begin{teo}[Main result]\label{teo:cact_sbk}
The map $\psi: s_{i,j} \mapsto \mathsf{q}_{i,j}$ is an epimorphism from $J_n$ to $\mathcal{SBK}_n$, for $1 \leq i < j \leq n$. Hence $\mathcal{SBK}_n$ is isomorphic to $J_n / \ker \psi$.
\end{teo}

\begin{proof}
\cref{lema:ti_si} ensures that $\mathsf{q}_i$ are generators for $\mathcal{SBK}_n$, for $i \in I$. Since $\mathsf{q}_i = \psi (s_{1,i})$ and thus $\mathsf{q}_{i,j} = \psi (s_{1,j-1}s_{1,j-i}s_{1,j-1})$, \eqref{eq:cactus_s1k} ensures that $\psi$ is a surjection. \cref{teo:qi_evaci} states that $\mathsf{q}_i$ acts as $\mathsf{evac}_{i+1}$ on straight-shaped tableaux, hence it follows from \cref{cor:cact_evaci} that $\psi$ is an homomorphism. Thus, $\mathcal{SBK}_n$ is isomorphic to the quotient of $J_n$ by $\ker \psi$.
\end{proof}

\cref{thm:rel_cact_bk}, which is stated in terms of group generators and not of specific operators, ensures that the relations in \cref{prop:rels_SBK} are equivalent to
$$\mathsf{q}_{i,j}^2 = 1, \qquad \mathsf{q}_{i,j} \mathsf{q}_{k,l} \mathsf{q}_{i,j} = \mathsf{q}_{i+j-l,i+j-k}, \;\text{for}\; i \leq k < l \leq j, \qquad \mathsf{q}_{i,j} \mathsf{q}_{k,l} = \mathsf{q}_{k,l} \mathsf{q}_{i,j}, \;\text{for}\; j <k.$$
This means that the generators $\mathsf{t}_i$, for $i \in I$, provide an alternative presentation for $J_n$:

\begin{equation}\label{eq:cact_prsnt}
J_n = \langle \mathsf{t}_{i},\; i \in I \; | \; \mathsf{t}_i^2 = 1, \mathsf{t}_i \mathsf{t}_j= \mathsf{t}_j \mathsf{t}_i, \;\text{if}\; |i-j| >1, (\mathsf{t}_i \mathsf{q}_{j,k})^2 =1, \;\text{for} \; i+1 < j < k \rangle.
\end{equation}

\acknowledgements{The author is grateful to her supervisors O. Azenhas and M. M. Torres, and acknowledges the hospitality of the Department of Mathematics of University of Coimbra. 
Thanks also to Seung-Il Choi and Thomas Lam for pointing to O. Azenhas the references  \cite{CNO17} and \cite{CGP16}, respectively, during the FPSAC18 and the 82nd Séminaire Lotharingien de Combinatoire. 
 The author is partially supported by the LisMath PhD program (funded by the Portuguese Science Foundation). This research was made within the activities of the Group for Linear, Algebraic and Combinatorial Structures of the Center for Functional Analysis, Linear Structures and Applications (University of Lisbon), and was partially supported by the Portuguese Science Foundation, under the project UIDB/04721/2020.}

{\small \bibliographystyle{siam}\bibliography{bibliography}}

\end{document}